\documentclass [a4paper,reqno,10pt] {amsart} 
\usepackage{amsmath,esint}
\usepackage{amssymb}
\usepackage{fullpage}
\usepackage{amsthm}
\usepackage{bbm}
\usepackage{dsfont}
\usepackage{color}
\usepackage[utf8]{inputenc}
\usepackage[english]{babel}
\usepackage{latexsym}
\usepackage{mathtools}
\usepackage{cite}
\usepackage{float}
\usepackage{enumerate}
\usepackage{hyperref}
\usepackage{parskip}
\def\N{\mathbb N}

\def\Q{\mathbb Q}

\def\R{\mathbb R}
    \renewcommand{\tau}{G}
      
    \renewcommand{\gamma}{{\mathfrak{m}}}

\newtheoremstyle{mystyle}
{}
{}
{\normalfont}
{}
{\normalfont\bfseries}
{}
{ }
{\textbf{\thmname{#1}\thmnumber{ #2}\thmnote{ (#3)}}}

\newtheorem{theorem}{Theorem}
\numberwithin{theorem}{section}
\newtheorem{lemma}[theorem]{Lemma}

\theoremstyle{mystyle}
\newtheorem{definition}[theorem]{Definition}

\newtheorem{remark}[theorem]{Remark}

\def\N{\mathbb N}

\def\Q{\mathbb Q}

\def\R{\mathbb R}
    \renewcommand{\tau}{G}
      
    \renewcommand{\gamma}{{\mathfrak{m}}}
    
\makeatletter

\pdfpageheight\paperheight
\pdfpagewidth\paperwidth

\makeatother

\begin{document}
\title{Limit theorems for counting large continued fraction digits}
\author[Kesseb\"ohmer]{Marc Kesseb\"ohmer}
  \address{Universit\"at Bremen, Fachbereich 3 -- Mathematik und Informatik, Bibliothekstr. 5, 28359 Bremen, Germany}
  \email{\href{mailto:mhk@math.uni-bremen.de}{mhk@math.uni-bremen.de}}
\author[Schindler]{Tanja Schindler}
  \address{Australian National University, Research School of Finance, Actuarial Studies and Statistics, 26C Kingsley St, Acton ACT 2600, Australia}
  \email{\href{mailto:tanja.schindler@anu.edu.au}{tanja.schindler@anu.edu.au}}
\thanks{This research was supported  by the German Research Foundation (DFG) grant {\em Renewal Theory and Statistics of Rare Events in Infinite Ergodic Theory} (Gesch\"aftszeichen KE 1440/2-1). TS was supported by the Studienstiftung des Deutschen Volkes.}
\begin{abstract}
We establish a central limit theorem for counting large continued fraction digits $(a_n)$, 
i.e.\ we count occurances $\{a_n>b_n\}$, where $(b_n)$ is a sequence of positive integers.
Our result improves a similar result by Philipp which additionally assumes
that $b_n$ tends to infinity.
Moreover, we give a refinement of the famous Borel-Bernstein Theorem for continued fractions
regarding the event that the $n$-th continued fraction digit lies infinitely often between $d_n$ and $d_n(1+1/c_n)$ for given sequences $(c_n)$ and $(d_n)$.
Also for these sets we  obtain a central limit theorem.
As an interesting side result we determine the first  $\phi$-mixing coefficient for the Gauss system explicitly.
\end{abstract}
\date{\today}

\maketitle

\section{Introduction and statement of results}\label{Sec: Introduction}
Throughout the paper, for any irrational number $\omega\in \R\setminus \Q$ we will denote its unique infinite regular continued fraction expansion by $\left[a_0(\omega);a_1({ \omega}),a_2({ \omega}),\ldots\right]$, where
\begin{align*}
{ \omega}= a_0({ \omega})+ \cfrac{1}{a_1({ \omega})+\cfrac{1}{a_2({ \omega})+\cfrac{1}{\ddots}}}\eqqcolon \left[a_0({ \omega});a_1({ \omega}),a_2({ \omega}),\ldots\right].
\end{align*}
  In case that $a_0({ \omega})=0$ we write ${ \omega}=\left[a_1({ \omega}),a_2({ \omega}),\ldots\right]$.
We may also  express this algorithm restricted to $I\coloneqq[0,1)$ by the Gauss map $\tau:I\rightarrow I$,   
\begin{align*}
\tau \left({ \omega}\right)\coloneqq \begin{cases}1/{ \omega}-\left\lfloor 1/{ \omega} \right\rfloor,& { \omega}\neq 0,\\ 0,&{ \omega}=0.\end{cases} 
\end{align*}
With $G^0\coloneqq \textup{id}$ and $G^n\coloneqq G\circ G^{n-1}$, $n\geq 1$ we obtain the sequence of digits  $a_n\left({ \omega}\right)\coloneqq \left\lfloor 1/\tau^{n-1}\left({ \omega}\right)\right\rfloor$, $n\in \N$. 
The algorithm will terminate in  $n\in \N$ only for rational numbers ${ \omega}$, whenever $G^n({ \omega})=0$ for the first time; in this way we obtain the finite continued fraction expansion  of ${ \omega}\in \Q$. 
  In this case we write ${ \omega}=\left[a_1({ \omega}),a_2({ \omega}),\ldots, a_n({ \omega})\right]$.

The transformation $\tau$ does not preserve the Lebesgue measure restricted to ${I}$ denoted by $\lambda_{ I}$ (cf. \cite[Chapter 1.3.3]{dajani_ergodic_2002}). However, Gauss found a $\tau$-invariant measure $\gamma$ which is equivalent to  $\lambda$ with density $h(x)={1}/(({x+1}){\log 2} )$, $x\in{ I}$ (cf. \cite[Chapter 1.2.2]{iosifescu_metrical_2009}). 
The dynamical system $({I},\mathcal{B}\lvert_{{ I}},\tau,\gamma)$ is in fact ergodic
and even a lot stronger mixing conditions hold, see for example \cite{bradley_basic_2005} and Section \ref{mixing}.

Based on these mixing properties Philipp proved a central limit theorem (CLT)
counting large entries of the continued fraction expansion, see \cite[Theorem 2]{philipp_metrical_1970}. 
He also gave a remainder term
which was later improved by Zuparov giving, \cite{zuparov_theorem_1981}.
Philipp and Webb \cite{philipp_invariance_1973},
improved the result by proving a functional limit theorem in the function space $\mathbb{D}[0,1]$.

To state this theorem by Philipp we let
$\mathbb{V}\left(X\right)\coloneqq \mathbb{E}\big(\left(X-\mathbb{E}\left(X\right)\right)^2\big)$
denote the variance of the random variable $X$.
\begin{theorem}[{\cite[Theorem 2]{philipp_metrical_1970}}]\label{thm: clt philipp}
Consider a sequence of positive reals $\left(b_n\right)_{ n\in\mathbb{N}}$.
If 
\begin{align*}
 \lim_{n\to\infty}b_n=\infty
 \,\,\,\,\text{ and }\,\,\,\,
\sum_{n\in\mathbb{N}}\frac{1}{b_n}=\infty,
\end{align*}
then for  $S_n\coloneqq \sum_{k=1}^n\mathbbm{1}_{\{a_k>b_k\}}$ we have
\begin{align*}
\lim_{n\to \infty}\gamma\left(\frac{S_{n}-\mathbb{E}\left(S_n\right)}{\sqrt{\mathbb{V}\left(S_n\right)}}<z\right)=\frac{1}{\sqrt{2\pi}}\int_{-\infty}^{z} e^{-t^{2}/2}\mathrm{d}t.
\end{align*}
\end{theorem}
The necessary conditions in this  theorem are intimately connected with the following classical 
zero-one law by Borel and Bernstein \cite{Bo1909, Be1912b, Be1912}.
\begin{theorem}[Borel-Bernstein Theorem] \label{thm: borel bernstein} 
Consider a sequence of positive reals $\left(b_n\right)_{ n\in\mathbb{N}}$. Then
$a_{n}(\omega) \geq b_n$ 
 holds infinitely often  with Lebesgue measure $0$ or $1$, according as the series $\sum _{n\in\N}{1/ b_n}$  converges or diverges. 
\end{theorem}
If one compares the two theorems stated above, 
the additional condition $\lim_{n\to\infty}b_n=\infty$ in the CLT seems to be artificial. 
We might also compare the assumptions of Theorem \ref{thm: clt philipp} 
with the necessary conditions for the CLT to hold in the case that
$(X_i)\coloneqq (\mathbbm{1}_{A_{i}})$ is a sequence  of independent  random variables 
(like for the L\"uroth system, cf.\ for example \cite{galambos_remarks_1972}) with the same distribution function. 
We can then make use of Lindeberg's condition.
That is we assume that  for all $\epsilon>0$
we have
\begin{align}
\lim_{n\to\infty}\frac{1}{\mathbb{V}\left(S_n\right)}\cdot\sum_{i=1}^n\mathbb{E}\left(\left(X_{i}-\mathbb{E}\left(X_i\right)\right)^2\cdot\mathbbm{1}_{\left\{\left|X_{i}-\mathbb{E}\left(X_i\right)\right|>\epsilon\cdot \sqrt{\mathbb{V}\left(S_n\right)}\right\}}\right)=0.\label{eq: Lindeberg allg}
\end{align}
We find that this condition is in fact equivalent to $\lim_{n\to\infty}\mathbb{V}\left(S_n\right)=\infty$
which in the i.i.d.\ case is equivalent to
$\lim_{n\to\infty}\sum_{i=1}^n\gamma\left(A_i\right)\cdot \gamma\left(A_i^c\right)=\infty$.
This can be seen as follows: 
On the one hand, the condition  $\lim_{n\to\infty}\mathbb{V}\left(S_n\right)=\infty$ 
combined with the fact that $X_i-\mathbb{E}\left(X_i\right)\in\left[-1,1\right]$
implies for all $\epsilon>0$ and  all $n$ sufficiently large 
that we have ${\big\{\left|X_{i}-\mathbb{E}\left(X_i\right)\right|>\epsilon\cdot \sqrt{\mathbb{V}\left(S_n\right)}\big\}}=\varnothing$; hence \eqref{eq: Lindeberg allg}   holds.
On the other hand, if $\lim_{n\to\infty}\mathbb{V}\left(S_n\right)<\infty$ and 
$\sum_{n=1}^{\infty}\gamma\left(A_n\right)=\infty$,
then there exists  $\epsilon>0$ sufficiently small such that  
\[
\mathbb{E}\left(\left(X_{i}-\mathbb{E}\left(X_i\right)\right)^2\cdot\mathbbm{1}_{\left\{\left|X_{i}-\mathbb{E}\left(X_i\right)\right|>\epsilon\cdot \sqrt{\mathbb{V}\left(S_n\right)}\right\}}\right)>0,
\]
for some $i\in \N$ and arbitrarily large $n\in\mathbb{N}$. Consequently, \eqref{eq: Lindeberg allg} fails to hold.

It is the purpose of this paper to prove the following general CLT for indicator functions 
measurable with respect to the continued fraction digits
requiring only the assumptions necessary for the i.i.d.\ case. 
\begin{theorem}\label{clt}
Let $\left(A_n\right)_{ n\in\mathbb{N}}$ be a sequence of events such that $A_n\in\sigma\left(a_n\right)$ for all $n\in\mathbb{N}$.
Suppose  
\begin{align}
V_n\coloneqq \sum_{k=1}^n\gamma\left(A_k\right)\cdot \gamma\left(A_k^c\right)\to\infty.\label{eq: mG Bn}
\end{align}
Then for  $S_n\coloneqq \sum_{k=1}^n\mathbbm{1}_{A_k}$
we have
\begin{align}
\lim_{n\to\infty}\gamma\left(\frac{S_{n}-\mathbb{E}\left(S_n\right)}{\sqrt{\mathbb{V}\left(S_n\right)}}<z\right)
=\frac{1}{\sqrt{2\pi}}\int_{-\infty}^{z} e^{-t^{2}/2}\mathrm{d}t.\label{eq: clt an>bn}
\end{align}
\end{theorem}
We remark here that  to provide error terms as in \cite{philipp_metrical_1970}
and \cite{zuparov_theorem_1981}, or to
prove that for $S_n$ a functional CLT as in \cite{philipp_invariance_1973} holds 
follows along the same lines as in the original papers and will be omitted.

From this theorem one can easily obtain the following improvement of Philipp's Theorem \ref{thm: clt philipp}.
\begin{theorem}\label{thm: clt an>bn}
Consider a sequence of positive reals $\left(b_n\right)_{n\in\mathbb{N}}$.
If 
\begin{align}
\sum_{n\in\mathbb{N}}\frac{1}{b_n}=\infty,\label{eq: sum bn}%\label{eq: mG Bn}
\end{align}
then for  $S_n\coloneqq \sum_{k=1}^n\mathbbm{1}_{\{a_k>b_k\}}$ we have
\begin{align*}
\lim_{n\to \infty}\gamma\left(\frac{S_{n}-\mathbb{E}\left(S_n\right)}{\sqrt{\mathbb{V}\left(S_n\right)}}<z\right)=\frac{1}{\sqrt{2\pi}}\int_{-\infty}^{z} e^{-t^{2}/2}\mathrm{d}t.
\end{align*}
\end{theorem}

One of the main ingredients for the proof of Theorem \ref{clt} is to prove that 
$\lim_{n\to\infty}\mathbb{V}\left(S_n\right)=\infty$ is equivalent to $\lim_{n\to\infty}\sum_{i=1}^n\gamma\left(A_i\right)\cdot\gamma\left(A_i^c\right)=\infty$.
We will do this by making use of the $\phi$- and $\psi$-mixing properties of the continued fraction digits, 
in particular we will give the precise value of the first $\phi$-mixing coefficient (Lemma \ref{lem: phi 1}) 
which might be of independent interest.
In Remark \ref{rem: phi(1)} we will make clear that this estimate of the first $\phi$-mixing coefficient
is a real improvement as the known estimates for the $\phi$-mixing coefficients 
are indeed not sufficient to prove the above equivalence.

\subsection{Central limit theorems and zero-one laws for intervals}\label{subsectionCLT}
The above stated theorems are all dealing with the sets $\{a_n>b_n\}$.
To broaden the picture we will consider in this section the more general
setting of $a_n$ hitting a certain interval. 
Results in this spirit have been proven for independent random variables 
in the context of L\"uroth expansions by Galambos in \cite{galambos_remarks_1972}.

We start with giving a CLT which can also be deduced from Theorem \ref{clt}.
\begin{theorem}\label{thm: clt all}
Let $\left(c_{n}\right)_{n\in\mathbb{N}}$ be an arbitrarily chosen sequence 
of positive real numbers and $\left(d_n\right)_{n\in\mathbb{N}}$ be a sequences of positive integers. 
Suppose that either 
\begin{enumerate}[\rm (A)]
%\item\label{en: clt 1} $A_n\coloneqq \left\{a_n\geq b_n\right\}$ with $\sum_{n:b_{n}>1} {1}/{b_n} =\infty,\label{eq: 1/cn =infty}$
\item\label{en: clt 2} $A_n\coloneqq {\left\{a_n= d_n\right\}}$ with $\sum_{n\in\mathbb{N}}{1}/{d_{n}^2}=\infty$,
\item\label{en: clt 3} $A_n\coloneqq \left\{d_n\leq a_n \leq d_n \cdot \left(1+\frac{1}{c_n}\right)\right\} $ 

with 
$\sum_{n:d_n>1}{1}/{\left(c_n d_n\right)}=\infty$ or $ \sum_{n:d_n>1}{1}/{d_n^2}=\infty$ or
$\sum_{n:d_n=1}c_n=\infty$,
\item\label{en: clt 4} $A_n\coloneqq \left\{d_n< a_n \leq d_n \cdot \left(1+\frac{1}{c_n}\right)\right\}$ with $ \sum_{n\colon c_n\leq d_n}{1}/{\left(c_n d_n\right)}=\infty$.
\end{enumerate}
Then for $S_n\coloneqq \sum_{k=1}^n\mathbbm{1}_{A_k}$ the CLT  in \eqref{eq: clt an>bn} holds.
\end{theorem}

In a similar way
we can give zero-one laws as an analog to the Borel-Bernstein theorem.
\begin{theorem}\label{cn dn}
Let $\left(c_n\right)_{n\in\mathbb{N}}$ be a sequence of positive real numbers 
 and $\left(d_n\right)_{n\in\mathbb{N}}$ be a sequences of positive integers.
Suppose that either  
\begin{enumerate}[\rm (A)]
%\item\label{en: clt 1} $A_n\coloneqq \left\{a_n\geq b_n\right\}$ with $\sum_{n:b_{n}>1} {1}/{b_n} =\infty,\label{eq: 1/cn =infty}$
\item\label{en: bb 2} $A_n\coloneqq {\left\{a_n= d_n\right\}}$ and $\Gamma\coloneqq\sum_{n\in\mathbb{N}}{1}/{d_{n}^2}$,
\item\label{en: bb 3} $A_n\coloneqq \left\{d_n\leq a_n \leq d_n \cdot \left(1+\frac{1}{c_n}\right)\right\} $ 
and $\Gamma\coloneqq \max \left\{\sum_{n\in\N}\frac{1}{c_n d_n},\sum_{n\in\N}\frac{1}{d_n^2}\right\}$,
\item\label{en: bb 4} $A_n\coloneqq \left\{d_n< a_n \leq d_n \cdot \left(1+\frac{1}{c_n}\right)\right\}$ and $\Gamma\coloneqq\sum_{n\colon c_n\leq d_n}{1}/{\left(c_n d_n\right)}$.
\end{enumerate} 
Then $A_n$
holds infinitely often with Lebesgue measure $0$ or $1$, according as $\Gamma$ is finite or not.
\end{theorem}

Regarding \eqref{en: bb 2} we remark that for $d_{n}\coloneqq \lfloor\sqrt{n\log (n)}\rfloor$ there are almost surely infinitely many values of $n$ such that $a_n=d_{n}$
and for $e_{n}\coloneqq \lfloor\sqrt{n}\log(n)\rfloor$ there are almost surely only finitely many values of $n$ such that $a_n=e_{n}$.
Particularly, if $\left(d_n\right)$ is bounded, then  almost surely  $a_n=d_{n}$ infinitely often.

Using a dynamical Borel-Cantelli lemma by Philipp, see \cite[Theorem 3]{philipp_metrical_1967}, 
the results can be proven in a similar way as in the i.i.d.\ case, see Section \ref{Sec: 0-1 laws}. 

The condition for a Lebesgue measure $1$ set in Theorem \ref{cn dn} 
differs in case \eqref{en: bb 3} from the condition for a CLT. 
Indeed, the condition for a zero-one law is $\sum_{n=1}^{\infty}\gamma\left(A_n\right)=\infty$,
see Section \ref{Sec: 0-1 laws} and particularly Lemma \ref{lem: 01 Philipp} which will be applied to prove the zero-one laws. 
In contrast to this, the condition for the CLT is $\sum_{n=1}^{\infty}\gamma\left(A_n\right)\cdot\gamma\left(A_n^c\right)=\infty$
and in this case we also have to ensure that $\gamma\left(A_n^c\right)$ remains large enough. 

A refined study of the Lebesgue measure zero sets may also be of interest.
In fact, in \cite{wang_hausdorff_2008} the Hausdorff dimension of $\left\{x\colon a_n\left(x\right)\geq b_n \text{ infinitely often}\right\}$
was subject of study for those $\left(b_n\right)$ for which the above set has $0$ Lebesgue measure.
These results could be carried over to the limsup  sets considered in this paper. 
We also note here that the Hausdorff dimension of similar sets, namely   
with certain restriction on the continued fraction digits has been widely studied,
for example for the set $\left\{x\colon a_n\in\left\{1,2\right\}\text{ for all }n\in\N\right\}$ in \cite{jenkinson_rigorous_2018} and previous works like \cite{TexanConjecture, KS}. 
Also sets concerning   restrictions  like
$\left\{x\colon s_n< a_n\left(x\right)\leq  s_nt_n \text{ for all }n\in\N\right\}$ with $\left(s_n\right), \left(t_n\right)$ being sequences of reals with $\left(s_n\right)$ tending to infinity 
have been studied in \cite{fan_khintchine_2009} and \cite{liau_upper_2016}   as a tool to determine the fast Khintchine spectrum.

\subsection{Khintchine's Theorem and zero-one laws for associated random variables}
Inspired by Theorem \ref{cn dn}
we will next state analogous results to Khintchine's  famous  zero-one law for Diophantine approximation which can be stated as follows \cite{khintchine_metrische_1935}.
\begin{theorem}[Khintchine's Theorem]\label{thm: khintchine}
Let $k:\N\to (0,\infty)$ be such that $\left(n\cdot k(n)\right)$, $n\in\N$ is non-increasing.  Then we have that for infinitely many $q\in\mathbb{N}$ there exists $p\in\mathbb{N}$ with greatest common divisor $(p,q)=1$ 
such that
\[
\left\vert x-\frac{p}{q}\right\vert\leq \frac{k(q)}{q}
\]
holds with Lebesgue measure $0$ or $1$, according as
$\sum_{n\in\mathbb{N}}k(n)$ 
is finite or not. 
\end{theorem}

Next, we define random variables that bridge the continued fraction digits  $\left(a_n\right)$ of an irrational number  to its Diophantine properties (see e.g.  \cite[Chapter 1.2.1]{iosifescu_metrical_2009}). 
\begin{lemma}\label{pn qn}
Fix $x\coloneqq [a_0;a_1,\ldots ]\in\mathbb{R}\setminus\Q$. Then with $p_{-1} \coloneqq1$, $p_0\coloneqq a_0$, $q_{-1}\coloneqq 0$, $q_0\coloneqq 1$,
\begin{align*}
p_n\coloneqq a_n p_{n-1}+p_{n-2}\text{, }\; q_n\coloneqq a_n q_{n-1}+q_{n-2}\;\mbox{ and  }\;
r_{n}\coloneqq \frac{1}{\tau^{n-1}}=\left[a_{n};a_{n+1},a_{n+2},\ldots\right]
\end{align*}
we have $q_{k}p_{k-1}-p_{k}q_{k-1}=\left(-1\right)^{k}$, 
\begin{align*}
x=\frac{p_{n-1}r_n+p_{n-2}}{q_{n-1}r_n+q_{n-2}}\; \mbox{ and }\;\;
 \frac{p_n}{q_n}&=a_0+\cfrac{1}{a_1+\cfrac{1}{a_2+\cfrac{1}{\ddots+\cfrac{1}{a_n}}}}= [a_0;a_1,\ldots,a_n].
\end{align*}
\end{lemma}
Setting 
 \begin{align}
y_{n}&\coloneqq \frac{q_{n}}{q_{n-1}},\label{yn}
\end{align}
for $n\in\N$, we have $q_n=y_1\cdots y_n$ and $y_n=\left[a_n;a_{n-1},\ldots,a_1\right]=a_n+y_{n-1}$, $n\in \N$.
Further define 
 \begin{align}
u_{n}&\coloneqq q_{n-1}^{-2}\left|x-\frac{p_{n-1}}{q_{n-1}}\right|^{-1}.\label{un}
\end{align}
The random variable $u_n$ is crucial in the context of Diophantine approximations. 
Recall the well-known estimate 
\begin{align}
\frac{1}{q_{n-1} \left(q_n+q_{n-1}\right)}<\left|x-\frac{p_{n-1}}{q_{n-1}}\right|<\frac{1}{q_{n-1}q_n}.\label{eq: q_n q_n+1}
\end{align}
For a comprehensive account we refer to \cite[Chapter 5]{dajani_ergodic_2002}, \cite{borwein_neverending_2014}, \cite{kesseboehmer_infinite_2016} or 
\cite{iosifescu_metrical_2009}.
Even more, Khintchine's Theorem \ref{thm: khintchine} can be reformulated in terms of continued fraction entries as follows.
\begin{lemma}
Let $k:\N\to (0,\infty)$ be such that $n\cdot k(n)$ is non-increasing. Then 
\[
\left\vert x-\frac{p_i}{q_i}\right\vert=\frac{1}{u_i\cdot q_i^2}\leq \frac{k(q_i)}{q_i}
\]
holds for infinitely many $i\in\mathbb{N}$ with Lebesgue measure $0$ or $1$, according as 
$\sum_{n\in\mathbb{N}}k(n)$ 
 is finite or not.  
\end{lemma}
This can be easily deduced following the proof of Khintchine's Theorem in \cite[Theorem 32]{khintchine_continued_1964}.

As we will see in the next lemma, the difference between the above defined variables and $a_n$ is bounded, which enables us to prove a theorem related to the continued fractions entries.
\begin{lemma}\label{r y u}
Let the random variables $\left(r_n\right)_{n\in\mathbb{N}}$, $\left(y_n\right)_{n\in\mathbb{N}}$, and $\left(u_n\right)_{n\in\mathbb{N}}$ be defined as above. Then
\begin{enumerate}[\rm(A)]
\item $a_n\leq r_n<a_n+1$,\label{r_n}
\item $a_n\leq y_n<a_n+1$,\label{y_n}
\item $a_n< u_n<a_n+2$.\label{u_n}
\end{enumerate}
\end{lemma}
\begin{proof}
The inequalities \eqref{r_n} and \eqref{y_n} are immediate, \eqref{u_n} follows from Eq. \eqref{eq: q_n q_n+1}.
\end{proof}

For the associated random variables $(r_n)$, $(y_n)$, and $(u_n)$ a Borel-Bernstein Theorem also holds. This follows immediately from an application of Lemma \ref{r y u} and \cite[Corollary 1.3.17]{iosifescu_metrical_2009}
to the Borel-Bernstein Theorem. 

Our next theorem will give an analogous statement to Theorem \ref{cn dn} for the associated random variables $(r_n)$, $(y_n)$, and $(u_n)$.

\begin{theorem}[Corollary to Theorem \ref{cn dn}]\label{thm: cor to thm cn dn}
 Let $(r_n)_{n\in\mathbb{N}}$, $(y_n)_{n\in\mathbb{N}}$, and $(u_n)_{n\in\mathbb{N}}$ be the random variables associated to the continued fraction digits, as defined in Lemma  \ref{pn qn} and Eq. \eqref{yn} and \eqref{un}.
 Further, let $\left(c_n\right)_{n\in\mathbb{N}}$ and $\left(d_n\right)_{n\in\mathbb{N}}$ be sequences of positive real numbers 
 such that there exists $\epsilon>0$ fulfilling $c_n\leq d_n/(3+\epsilon)$ for all $n\in\mathbb{N}$ in case that we consider $r_n$ or $y_n$ and $c_n\leq d_n/(4+\epsilon)$ for all $n\in\mathbb{N}$ in case that we consider $u_n$. 
 Then each of the corresponding inequalities
\begin{align*}
d_n\leq r_n\leq d_n\left(1+1/c_n\right),\,\,\,\,\,\,\,\,\,\,d_n\leq y_n\leq d_n\left(1+1/c_n\right),\,\,\,\,\,\,\,\,\,\,d_n\leq   u_n\leq d_n\left(1+1/c_n\right)
\end{align*}
holds infinitely often with Lebesgue measure $0$ or $1$, according as 
$\sum_{n\in\N}{1}/({c_n d_n})$
is finite or not. 
\end{theorem}
\begin{remark}
 Other than in Theorem \ref{cn dn} we have here the additional condition that $c_n\leq d_n/(3+\epsilon)$ or $c_n\leq d_n/(4+\epsilon)$ respectively.
 Proving Theorem \ref{thm: cor to thm cn dn} as a corollary to Theorem \ref{cn dn} with the help of 
 Lemma \ref{r y u} makes this restriction necessary. 
 Following this proof one might see that the condition can be relaxed to $c_n\leq d_n/(1+\epsilon)$ or $c_n\leq d_n/(2+\epsilon)$ respectively if we require $d_n\in\mathbb{N}$.
 However, it would be interesting if one could state the above theorem also for comparatively small intervals $\left[d_n, d_n(1+1/c_n)\right]$, i.e.\ intervals with large $c_n$.
\end{remark}

%\subsection{Outline of the paper}
%We will first introduce the notion of $\phi$- and $\psi$-mixing
%in Section \ref{mixing}.
%give an overview of the natural extension of the continued fraction expansion in Section \ref{Sec: ext rand var}. 
%In the following Section \ref{mixing} we introduce 
%We also derive the exact value of the first $\phi$-mixing coefficient $\phi_{\gamma}\left(1\right)$ of the Gauss system using the extended random variables,
%improving a result of Philipp  \cite[Lemma 2.1]{philipp_limit_1988} who showed that $\phi_{\gamma}\left(1\right)<0.4$.
%
%Finally, in Section \ref{Sec: 0-1 laws} we give a proof of the zero-one laws and in Section \ref{sec: Proof CLT} we give the proof of the CLT,
%giving an explanation at the beginning of the chapter why the exact value of $\phi_{\gamma}\left(1\right)$ is needed 
%for an improvement of the Philipp blub result
%and exponentially $\psi$-mixing could not suffice.

\section{Mixing properties}\label{mixing}
Our results will depend crucially on the mixing properties of the continued fraction digits. To explain this we first introduce the classical notion of $\phi$- and $\psi$-mixing.
\begin{definition} 
Let $\left(\Omega,\mathcal{A},\mathbb{P}\right)$ be a probability space and $\mathcal{C},\mathcal{D}\subset\mathcal{A}$ two $\sigma$-fields, then the following quantities measure the dependence of the sub-$\sigma$-fields.
\begin{align*}
 \phi\left(\mathcal{C},\mathcal{D}\right)&\coloneqq \sup_{\genfrac{}{}{0pt}{1}{C\in\mathcal{C}, D\in\mathcal{D}}{\mathbb{P}\left(C\right)>0}}\left|\mathbb{P}\left(D\mid C\right)-\mathbb{P}\left(D\right)\right|
 \,\,\,\,\,\,\,\,\,\,\,\,\,\,\,\,\,\,\text{ and }\,\,\,\,\,\,\,\,\,\,\,\,\,\,\,\,\,\,
\psi\left(\mathcal{C},\mathcal{D}\right)\coloneqq \sup_{\genfrac{}{}{0pt}{1}{C\in\mathcal{C}, D\in\mathcal{D}}{ \mathbb{P}\left(C\right), \mathbb{P}\left(D\right)>0}}\left|\frac{\mathbb{P}\left(C\cap D\right)}{\mathbb{P}\left(C\right)\cdot \mathbb{P}\left(D\right)}-1\right|.
\end{align*}

Let $\left(X_{n}\right)_{n\in\mathbb{N}}$ be a (not necessarily stationary) sequence of random variables. For $0\leq J\leq L\leq\infty$ we can define a $\sigma$-field by
\[\mathcal{A}_{J}^{L}\coloneqq \sigma\left(X_{k}, k\in\mathbb{N}\cap[J,L]\right).
\]
With that the dependence coefficients are defined by
\begin{align*}
\phi\left(n\right)&\coloneqq \sup_{k\in{ \mathbb{Z}}}\phi\left(\mathcal{A}_{ -\infty}^{k},\mathcal{A}_{k+n}^{\infty}\right)
\,\,\,\,\,\,\,\,\,\,\,\,\,\,\,\,\,\,\text{ and }\,\,\,\,\,\,\,\,\,\,\,\,\,\,\,\,\,\,
\psi\left(n\right)\coloneqq \sup_{k\in{ \mathbb{Z}}}\psi\left(\mathcal{A}_{ -\infty}^{k},\mathcal{A}_{k+n}^{\infty}\right).
\end{align*}
The sequence $\left(X_{n}\right)$ is said to be  $\phi$-mixing or 
$\psi$-mixing if
$\phi(n)\rightarrow 0$ or
$\psi(n)\rightarrow 0$ as $n\rightarrow\infty$. 
\end{definition} 
It follows easily that 
\begin{align}
\phi\left(n\right)\leq \frac{1}{2}\psi\left(n\right),\label{eq: phi psi} 
\end{align}
for all $n\in\mathbb{N}$, see also \cite[(1.11)]{bradley_basic_2005}.
For more details about mixing conditions see \cite{bradley_basic_2005}.

Now we collect the necessary mixing properties of the continued fractions digits. 
We start by stating the following lemma from 
\cite[Chapter 2.3.4]{iosifescu_metrical_2009}.
\begin{lemma}\label{lem: mixing cf}
Let $\psi=\psi_{\gamma}$ denote the $\psi$-mixing coefficient with respect to the continued fraction digits and the  Gauss measure $\gamma$. 
Then we have that
\[
\psi_{\gamma}\left(n\right)\leq \rho\, \theta^{n-2}\text{ for } n\geq 2,
\]
where $\rho=\pi^2\log 2/6-1$ and $\theta$ is a constant less than $0.30367$, and $\psi_\gamma\left(1\right)=2\log2-1$, i.e.\ the digits of the continued fraction expansion are exponentially $\psi$-mixing.
\end{lemma}

Next we state the exact value of the $\phi$-mixing coefficient.
\begin{lemma}\label{lem: phi 1}
Let $\phi = \phi_{\gamma}$ denote the  $\phi$-mixing coefficient for the Gauss system.
Then we have that 
\begin{align*}
\phi_{\gamma}\left(1\right)={ -}\frac{1-\log2+\log\log 2}{\log 2}<0.0861.
\end{align*}
\end{lemma}
This lemma improves a result of Philipp  \cite[Lemma 2.1]{philipp_limit_1988} who showed that $\phi_{\gamma}\left(1\right)<0.4$.
We also remark here that the value of $\phi_{\gamma}\left(1\right)$ coincides with the Erd\H{o}s-Ford-Tenenbaum constant, see for example \cite{ford_distribution_2008}.

The remaining part of this section is devoted to the proof of Lemma \ref{lem: phi 1}.
One main ingredient for this is the natural extension of the Gauss measure 
and an associated auxiliary family of measures to be introduced next.

The basic idea is to construct a doubly infinite version of $\left(a_n\right)_{n\in\mathbb{N}}$ 
under $\gamma$ for which we use the natural extension. We first define $\overline{\tau}\colon I^2 \to I^2$ by
\begin{align*}
 \overline{\tau}\left(\omega,\theta\right)\coloneqq \left(\tau\left(\omega\right),\frac{1}{a_1\left(\omega\right)+\theta}\right).
\end{align*}
It can be easily seen that 
\begin{align*}
\overline{\tau}^n\left(\omega,\theta\right)=\left(\tau^n\left(\omega\right), \left[a_{n}\left(\omega\right),\ldots,a_2\left(\omega\right),a_1\left(\omega\right)+\theta\right]\right), n\in\N.
\end{align*}

Inhere we slightly abuse notation since $a_1\left(\omega\right)+\theta$ is usually not a natural number.
Then we define the bi-infinite sequence $\left(\overline{a}_k\right)_{k\in\mathbb{Z}}$, where each $\overline{a}_k: I^2 \to\mathbb{N}$ is given by 
\begin{align*}
\overline{a}_k\left(\omega,\theta\right)\coloneqq \overline{a}_1\left(\overline{\tau}^k\left(\omega,\theta\right)\right)
\end{align*}
with
$\overline{a}_1\left(\omega,\theta\right)\coloneqq a_1\left(\omega\right)=\left\lfloor 1/\omega\right\rfloor
$. 

Furthermore, we define the extended Gauss measure $\overline{\gamma}$ for $B\in\mathcal{B}_{ I^2}$ by
\begin{align*}
\overline{\gamma}\left(B\right)\coloneqq \frac{1}{\log2}\cdot  \iint\limits_B \frac{1}{\left(xy+1\right)^2}\mathrm{d}x\mathrm{d}y,
\end{align*}
which is $\overline{G}$ invariant, see for example \cite[Theorem 1.3.4]{iosifescu_metrical_2009}.
In the following we give one lemma concerning the conditional distribution which are essential in the proof of Lemma \ref{lem: phi 1}.
\begin{lemma}[Theorem 1.3.5 of \cite{iosifescu_metrical_2009}]
For any $v, x\in I$ we have for the conditional probability 
\begin{align*}
\overline{\gamma}\left(\left(\omega,\theta\right)\in\left[0,x\right]\times  I \lvert(\overline{a}_0\left(\omega,\theta\right)=a_1\left(v\right),\overline{a}_{-1}\left(\omega,\theta\right)=a_2\left(v\right),\ldots)\right)=\frac{\left(v+1\right)x}{vx+1}\text{ }\overline{\gamma}\text{-a.s.}
\end{align*}
\end{lemma}
Motivated by this lemma we also define the  probability measure $\gamma_{ v}$ on $\mathcal{B}_{I}$ via its distribution function, for ${ v}\in { I}$, by
\begin{align}
\gamma_{ v}\left(\left[0,x\right]\right)\coloneqq \frac{\left({ v}+1\right)x}{vx+1}.\label{ma}
\end{align}
For further investigations of the natural extension of $\left(a_n\right)$ see \cite[Section 1.3]{iosifescu_metrical_2009}.

With this techniques at hand we are now able to begin with the proof of Lemma \ref{lem: phi 1}.
\begin{proof}[Proof of Lemma \ref{lem: phi 1}]
Let $\gamma_{ v}$ be the measure defined in \eqref{ma} and  let
\begin{align*}
\eta\coloneqq \sup\left\{\left|\gamma_{ v}\left(B\right)-\gamma\left(B\right)\right|, { v}\in I, B\in\mathcal{B}_I\right\}.
\end{align*}
The proof of the lemma is separated into two parts, namely we show that
\begin{enumerate}[\rm(A)]
 \item\label{en: theta 1} $\eta=-\left(1-\log2+\log\log 2\right)/\log 2$ and
 \item\label{en: theta 3} $\phi_{\gamma}\left(1\right)= \eta$.
\end{enumerate}
The proof of \eqref{en: theta 3} is inspired by the proof for the first $\psi$-mixing coefficient in \cite{iosifescu_metrical_2009}.

{\em ad} \eqref{en: theta 1}:
Let us define $f: I^2\to\mathbb{R}$ by
\begin{align*}
f\left({ v},x\right)\coloneqq  \gamma_{ v}\left(\left[0,x\right]\right)-\gamma\left(\left[0,x\right]\right)=\frac{\left({ v}+1\right)x}{{ v}x+1}-\frac{\log\left(x+1\right)}{\log 2}.
\end{align*}
We have that $f\left(v,\cdot\right)$ is the distribution function of a signed measure with density $\partial f\left({ v},x\right)/\partial x$. 
For each ${ v}\in I$ we obtain that $\sup_{\mathcal{B}_{ I}} (\gamma_{ v}\left(B\right)-\gamma\left(B\right))$ will be attained for $B=\left\{x\colon \partial f\left({ v},x\right)/\partial x>0\right\}$ and
$ \inf_{\mathcal{B}_{I}} \gamma_{ v}\left(B\right)-\gamma\left(B\right)$ will be attained for $B^c$. 
In the following we will only calculate $ \inf_{\mathcal{B}_{I}} \gamma_{ v}\left(B\right)-\gamma\left(B\right)$ since 
\begin{align*}
\gamma_{ v}\left(B^c\right)-\gamma\left(B^c\right)=1-\gamma_{ v}\left(B\right)-\left(1-\gamma\left(B\right)\right)=-\left(\gamma_{ v}\left(B\right)-\gamma\left(B\right)\right)
\end{align*}
and thus $\sup_{v\in  I} \gamma_{ v}\left(B\right)-\gamma\left(B\right)=- \inf_{ v\in I} \left(\gamma_{ v}\left(B\right)-\gamma\left(B\right)\right)$.

In the next steps we calculate the zeros of $\partial f\left( v,x\right)/\partial x$ depending on $v$. We have that 
\begin{align*}
\frac{\partial f\left(v,x\right)}{\partial x}=\frac{v+1}{\left(vx+1\right)^2}-\frac{1}{\log 2\cdot\left(x+1\right)}. 
\end{align*}
From  this we find that the two zeros are given for $v>0$ by 
\begin{align*}
x_{v,1}&=\frac{\left({ v}+1\right)\cdot \log2-2{ v}}{2{ v}^2}+\sqrt{\left(\frac{\left({ v}+1\right)\cdot \log2-2{ v}}{2{ v}^2}\right)^2-\frac{1-\left({ v}+1\right)\cdot\log2}{{ v}^2}} \text{ and }\\
x_{v,2}&=\frac{\left({ v}+1\right)\cdot \log2-2{ v}}{2{ v}^2}-\sqrt{\left(\frac{\left({ v}+1\right)\cdot \log2-2{ v}}{2{ v}^2}\right)^2-\frac{1-\left({ v}+1\right)\cdot\log2}{{ v}^2}}.
\end{align*}
If ${ v}=0$ we have that $\partial f\left({ v},x\right)/\partial x=1-1/\left(\log 2\cdot\left(x+1\right)\right)$ and we obtain that 
$x_{0,2}=1/\log 2-1$. (Taking limits for $x_{{ v},1}$ would lead to the degenerate value of $x_{0,1}=\infty$.)

We obtain as values of interest
\begin{align*}
 x_{2\log 2-1,1}&=1,&x_{1,1}&=2\log 2-1,&x_{1/\log 2-1,2}&=0,&x_{0,2}&=1/\log 2-1.
\end{align*}
In the next steps we will show 
that $x_{{ v},1}\in  \left[0,1\right]$ if and only if ${ v}\in\left[2\log2-1,1\right]$ and $x_{{ v},2}\in \left[0,1\right]$ if and only if ${ v}\in\left[0,1/\log2-1\right]$
using monotonicity.

We set $g\left({ v}\right)\coloneqq\left(\left({ v}+1\right)\cdot \log2-2{ v}\right)/2{ v}^2$ and $h\left({ v}\right)\coloneqq\left(1-\left({ v}+1\right)\cdot\log2\right)/{ v}^2$ such that
\begin{align*}
x_{{ v},1}=g\left({ v}\right)+\sqrt{g\left({ v}\right)^2-h\left({ v}\right)}
\,\,\,\,\,\,\,\,\,\,\,\,\text{ and }\,\,\,\,\,\,\,\,\,\,\,\,
x_{{ v},2}=g\left({ v}\right)-\sqrt{g\left({ v}\right)^2-h\left({ v}\right)}.
\end{align*}
We have that
\begin{equation}
 \frac{\partial g({ v})}{\partial { v}}=\frac{2-\log 2}{2{ v}^2}-\frac{\log 2}{{ v}^3}\text{ and }\frac{\partial h({ v})}{\partial { v}}=\frac{\log 2}{{ v}^2}-\frac{2-2\log 2}{{ v}^3}.\label{eq: part deriv}
\end{equation}
First we find that $\partial g({ v})/\partial { v}<0$ and $\partial h({ v})/\partial { v}<0$ for all ${ v}\in I$. 
To consider $x_{{ v},1}$ we notice that $x_{{ v},1}\geq 1$ if $g({ v})\geq 1$ and in this case it has to hold that ${ v}<2\log 2-1$.
Let us now consider $g({ v})<1$. We will show that in this case 
$x_{{ v},1}$ is monotonically decreasing in ${ v}$. Since we have shown that $g$ is monotonically decreasing we only have to show that 
$\sqrt{g^2-h}$ is monotonically decreasing which by the strict monotonicity of the square root is equivalent to the statement that 
$\partial g^2({ v})/\partial { v}-\partial h({ v})/\partial { v}<0$. By applying the chain rule for derivatives this is equivalent to
$2\cdot g({ v})\cdot \partial g( v)/\partial { v}<\partial h( v)/\partial  v$. 
Since we are assuming that $g( v)<1$ this already holds if 
$2\cdot \partial g( v)/\partial  v<\partial h( v)/\partial  v$. 
Using the derivatives in \eqref{eq: part deriv} and noticing that $2-\log 2<2\log 2$ and $\log 2>2-2\log 2$ yields that the last statement indeed is true.

In the next steps we consider $x_{{ v},2}$. If $g({ v})<0$, it immediately follows that $x_{{ v},2}\leq 0$ and by the fact that $g$ is monotonically decreasing this only happens for ${ v}>1/\log 2-1$. 
Let us now assume that $g({ v})< 0$. $x_{{ v},2}$ is monotonically decreasing if and only if
$\partial g({ v})/\partial { v}<\partial \sqrt{g^2({ v})-h({ v})}/\partial { v}$ which under the assumption that $g({ v})< 0$
is equivalent to $\partial \sqrt{g^2({ v})}/\partial { v}<\partial \sqrt{g^2({ v})-h({ v})}/\partial { v}$.
Since the square root is a strictly monotonic function this is equivalent to
$\partial g^2({ v})/\partial { v}<\partial \left(g^2({ v})-h({ v})\right)/\partial { v}$ and thus $\partial h({ v})/\partial { v}<0$
which by our previous notice holds true.

We have that $\partial f\left({ v},x\right)/\partial x$ changes sign from plus to minus in $x=x_{{ v},1}$ for ${ v}\in\left[2\log2-1,1\right]$ and $\partial f\left({ v},x\right)/\partial x$ changes sign from minus to plus in $x=x_{{ v},2}$ for ${ v}\in\left[0,1/\log2-1\right]$.

We consider in the following three cases, namely
\begin{enumerate}[(a)]
 \item\label{en: a1} $0\leq { v}< 2\log2-1$,
 \item\label{en: a2} $2\log2-1\leq { v}\leq 1/\log2-1$, and
 \item\label{en: a3} $1/\log2-1<{ v}\leq 1$.
\end{enumerate}

{\em ad \eqref{en: a1}}: In this case we have that $\inf_{B\in\mathcal{B}_{ I}}\gamma_{ v}\left(B\right)-\gamma\left(B\right)$ will be attained for $B=\left[0,x_{{ v},2}\right]$.

By determining the partial derivative of $f$ with respect to ${ v}$ 
\begin{align*}
\frac{\partial f\left({ v},x\right)}{\partial { v}}=\frac{x\cdot\left(1-x\right)}{\left({ v}x+1\right)^2}
\end{align*}
we obtain that for all $x\in { I} $ we have that $\partial f\left({ v},x\right)/\partial { v}\geq 0$, i.e.\ for all $x\in { I} $, $f$ is monotonically increasing in ${ v}$.
By the fact that $x_{{ v},2}$ is monotonically decreasing in ${ v}$ on the relevant parts we obtain for ${ v}>{ w}$ that
\begin{equation}
\gamma_{ v}\left(\left[0,x_{{ v},2}\right]\right)-\gamma\left(\left[0,x_{{ v},2}\right]\right)
=f\left({ v}, x_{{ v},2}\right)\geq f\left({ w}, x_{{ v},2}\right)\geq f\left({ w}, x_{{ w},2}\right).\label{eq: use of monotonicity} 
\end{equation}
Thus, $\inf_{{ v}\in\left[0,2\log2-1\right), B\in\mathcal{B}_{I }}\gamma_{ v}\left(B\right)-\gamma\left(B\right)=\inf_{B\in\mathcal{B}_{ { I} }}\gamma_0\left(B\right)-\gamma\left(B\right)$.
Using  $x_{0,2}=1/\log 2-1$ we find
\begin{align*}
f\left(0,\frac{1}{\log 2}-1\right)=\frac{1-\log2+\log\log2}{\log2} 
\end{align*}
and consequently
\begin{align}
\inf_{v\in\left[0,2\log2-1\right), B\in\mathcal{B}_{ { I} }}\gamma_{ v}\left(B\right)-\gamma\left(B\right)&=\frac{1-\log2+\log\log2}{\log2}.\label{eq: estim eta a}
\end{align}

{\em ad} \eqref{en: a2}: In this case we have that $\inf_{B\in\mathcal{B}_{ { I} }}(\gamma_{ v}\left(B\right)-\gamma\left(B\right))$ will be attained for $B=\left[0, x_{{ v},2}\right){ \cup}{ \left[x_{{ v},1},1\right)}$.
Furthermore, the monotonicity of $f$ in $v$ and a similar argument as in \eqref{eq: use of monotonicity} yields
\begin{align}
\MoveEqLeft  \inf_{ v\in \left[2\log2-1, 1/\log2-1\right], B\in\mathcal{B}_{ { I} }}\gamma_{ v}\left(B\right)-\gamma\left(B\right)\notag\\
&\geq \inf_{v\in \left[2\log2-1, 1/\log2-1\right]}\gamma_{ v}\left(\left[0,x_{{ v},2}\right)\right)-\gamma\left(\left[0,x_{{ v},2}\right)\right)\notag\\
&\qquad+\inf_{w\in \left[2\log2-1, 1/\log2-1\right]}\gamma_{ w}\left({ \left[x_{{ w},1},1\right)}\right)-\gamma\left({ \left[x_{{ w},1},1\right)}\right)\notag\\
&=\inf_{v\in \left[2\log2-1, 1/\log2-1\right]}f\left({ v},x_{{ v},2}\right)+\inf_{{ w}\in \left[2\log2-1, 1/\log2-1\right]}-f\left({ w},x_{{ w},1}\right).\label{eq: a2 min}
\end{align}

For the first summand in \eqref{eq: a2 min} we have by the monotonicity of $f$ in the first argument that the { infimum} will be attained for ${ v}=2\log2-1$ in $x_{2\log2-1,2}$ and for the second summand in \eqref{eq: a2 min} we again have by the monotonicity of $f$ in the first argument that the infimum will be attained for ${ v}=1/\log2-1$ in  $x_{1/\log2-1,1}$,
i.e.\ we have
\begin{align}
\inf_{ v\in \left[2\log2-1, 1/\log2-1\right]} f\left({ v},x_{v,2}\right)&=f\left(2\log2-1,x_{2\log2-1,2}\right)\text{ and}\label{eq: a2 min1 num}\\
\inf_{ w\in \left[2\log2-1, 1/\log2-1\right]} -f\left({ w},x_{w,1}\right)&=-f\left(1/\log2-1,x_{1/\log2-1,1}\right)
=f\left(2\log2-1,x_{2\log2-1,2}\right).\label{eq: a2 min2 num} 
\end{align}
Combining \eqref{eq: a2 min} with \eqref{eq: a2 min1 num} and \eqref{eq: a2 min2 num} yields
\begin{align}
\inf_{v\in \left[2\log2-1, 1/\log2-1\right], B\in\mathcal{B}_{ { I} }}\gamma_{ v}\left(B\right)-\gamma\left(B\right)\geq 2\cdot f\left(2\log2-1,x_{2\log2-1,2}\right)\geq -0.0118.\label{eq: estim eta b}
\end{align}

{\em ad} \eqref{en: a3}: In this case we have that ${ \inf}_{B\in\mathcal{B}_{ { I} }}(\gamma_{ v}\left(B\right)-\gamma\left(B\right))$ will be attained for $B=\left[x_{{ v},1},1\right)$.
Furthermore, since $f\left({ v},1\right)=0$, and by the { monotonicity} of $f$, using the argument as in \eqref{eq: use of monotonicity} in the reverse direction, we have that
\begin{align}
 \inf_{{ v}\in{ \left(1/\log2-1,1\right)}}\gamma_{ v}\left({ \left[x_{{ v},1},1\right)}\right)-\gamma\left(\left[x_{{ v},1},1\right]\right)
&= \inf_{{ v}\in\left(1/\log2-1,{ v}\right]}-f\left({ v},x_{{ v},1}\right)\notag\\
&=-f\left(1,x_{1,1}\right)\notag\\
&=-f\left(1,2\log 2-1\right)\notag\\
&=\frac{1-\log2+\log\log2}{\log2}.\label{eq: estim eta c}
\end{align}
Putting \eqref{eq: estim eta a}, \eqref{eq: estim eta b}, and \eqref{eq: estim eta c} together and noticing that $-0.0118>\left(1-\log2+\log\log2\right)/\log2$ yields the first statement. 

{\em ad} \eqref{en: theta 3}:
To prove the second part of the lemma we make use of the natural extension of the Gauss system. We have that 
\begin{align*}
\phi_{\gamma}\left(1\right)= \sup\left\{\left|\overline{\gamma}\left(\overline{B}\lvert \overline{A}\right)-\overline{\gamma}\left(\overline{B}\right)\right|\colon \overline{B}\in \sigma\left(\overline{a}_n, \overline{a}_{n+1},\ldots\right), \overline{A}\in \sigma\left(\ldots ,\overline{a}_{n-2}, \overline{a}_{n-1}\right), \overline{\gamma}\left(\overline{A}\right)>0 ,n\in\mathbb{N}\right\}.
\end{align*}
This follows directly from the definition of the bi-infinite sequence $\left(\overline{a}_n\right)_{n\in\mathbb{Z}}$ and the definition of the $\phi$-mixing coefficients.
By the shift invariance of $\gamma$ the supremum does not change if we restrict $n$ to $1$ and we thus only consider
$\overline{B}\in \sigma\left(\overline{a}_1, \overline{a}_{2},\ldots\right)$
and $\overline{A}\in \sigma\left(\overline{a}_0, \overline{a}_{-1},\ldots\right)$ for which $\overline{\gamma}\left(\overline{A}\right)>0$.

Clearly, $\overline{B} = B \times  I$ and $\overline{A} =I\times A$,
for some $A, B \in  \mathcal{B}_I$. 
Thus,
\begin{align}
\phi_{\gamma}\left(1\right)=\sup\left\{\left|\frac{\overline{\gamma}\left(A\times B\right)}{\gamma\left(A\right)}-\gamma\left(B\right)\right|\colon A, B\in \mathcal{B}_{I},\gamma\left(A\right)>0 \right\}.\label{eq: theta bar gamma gamma}
\end{align}
Furthermore, we have that
$\overline{\gamma}\left(A\times B\right)
={\int_A\gamma_{ v}\left(B\right)\mathrm{d}\gamma\left({ v}\right)},
$ 
for  $A,B\in\mathcal{B}_{I}$. 
For given $B\in\mathcal{B}_I$ we have that
\[
\sup_{{ v}\in I}\gamma_{ v}\left(B\right)\geq \sup_{A\in\mathcal{B}_I}\frac{{\int_A\gamma_{ v}\left(B\right)\mathrm{d}\gamma\left({ v}\right)}}{\gamma\left(A\right)}\text{ and }
\inf_{{ v}\in I}\gamma_{ v}\left(B\right)\leq \inf_{A\in\mathcal{B}_I}\frac{{\int_A\gamma_{ v}\left(B\right)\mathrm{d}\gamma\left({ v}\right)}}{\gamma\left(A\right)}.
\]
In the next steps we show that we have
equality if $\gamma_{ v}\left(B\right)$ is continuous in $v$ for given $B$. In this case 
the supremum is attained on $\left[0,1\right]$, say on a point $x_B$ and we have that
\[
\sup_{A\in\mathcal{B}_I}\frac{{\int_A\gamma_{ v}\left(B\right)\mathrm{d}\gamma\left({ v}\right)}}{\gamma\left(A\right)}
=\lim_{\epsilon\searrow 0}\frac{\int_{[x_B-\epsilon,x_B+\epsilon]\cap\left[0,1\right]}{\gamma_{ v}\left(B\right)\mathrm{d}\gamma\left({ v}\right)}}{\gamma\left([x_B-\epsilon,x_B+\epsilon]\cap\left[0,1\right]\right)}
=\sup_{{ v}\in I}\gamma_{ v}\left(B\right) 
\]
and analogously for the infimum. 
On the other hand we have by our calculation from \eqref{en: theta 1} that 
\begin{align*}
  \sup\left\{\left|\gamma_{v}\left(B\right)-\gamma\left(B\right)\right|, v\in I, B\in\mathcal{B}_I\right\}
 &=\left|\gamma_0\left(\left[0,1-1/\log2-1\right]\right)-\gamma\left(\left[0,1-1/\log2-1\right]\right)\right|\\
 &=\left|f\left(0,\left[0,1/\log2-1\right]\right)\right|
\end{align*}
and $f\left(v,\left[0,1/\log2-1\right]\right)$ is continuous in $v$. 
Now the claim in \eqref{en: theta 3} follows from this argument, \eqref{eq: theta bar gamma gamma} and the definition of $\eta$ by  noting that  
\begin{align*}
\phi_{\gamma}\left(1\right) & =\sup\left\{\left|\frac{{\int_A\gamma_{v}\left(B\right)\mathrm{d}\gamma\left({v}\right)}}{\gamma\left(A\right)}-\gamma\left(B\right)\right|\colon A, B\in \mathcal{B}_{I},\gamma\left(A\right)>0 \right\}\\
& = \sup\left\{\left|\gamma_{ v}\left(B\right)-\gamma\left(B\right)\right|\colon{ v}\in {I} , B\in\mathcal{B}_{I} \right\}=\eta.
\end{align*}
\end{proof}

\section{Proofs of  the zero-one laws}\label{Sec: 0-1 laws}
All the zero-one laws can be proven by the following lemma
which is a simplified version of \cite[Theorem 3]{philipp_metrical_1967}.
\begin{lemma}[{{\cite[Theorem 3]{philipp_metrical_1967}}}]\label{lem: 01 Philipp}
 Let $(\Gamma_n)_{n\in\mathbb{N}}$ be a sequence of measurable sets in
a probability space $\left(\Omega, \mathcal{A},\mu\right)$. 
Suppose that there exists a function $q\colon\mathbb{N}\to\mathbb{R}_{\geq 0}$ 
fulfilling $\sum_{n=1}^{\infty}q(n)<\infty$
such that for all integers $n > m$ we have
\begin{align*}
 \mu\left(\Gamma_m\cap\Gamma_n\right)
 \leq \mu\left(\Gamma_m\right)\cdot \mu\left(\Gamma_n\right)+q\left(n-m\right)\cdot \mu\left(\Gamma_n\right).
\end{align*}
Then $\Gamma_n$
holds infinitely often with Lebesgue measure
$0$ or $1$ according as $\sum_{n=1}^{\infty}\mu\left(\Gamma_n\right)$ is finite or not.
\end{lemma}

Using Lemma \ref{lem: mixing cf} we immediately obtain the following lemma, 
which enables us to prove Theorem \ref{cn dn} similarly as in the i.i.d.\ setting.
\begin{lemma}\label{lem: 0-1 mixing}
Let $\left(A_n\right)$ be a sequence of events such that $A_n\in\sigma\left(a_n\right)$ 
for all $n\in\mathbb{N}$. If $\sum_{n\in\mathbb{N}}\gamma\left(A_n\right)=\infty$, 
then $\gamma\left(\limsup_{n\to\infty}A_n\right)=1$. 
\end{lemma}

\begin{proof}[Proof of Theorem \ref{cn dn}]
Since $x\cdot\log\left(2\right)   \leq \log\left(1+x\right)\leq x$ for all $x\in [0,1)$,
it follows that for all considered sets $A_n$
\begin{align}
 \frac{\lambda\left(A_n\right)}{\log 2}
 \leq \gamma\left(A_n\right)
 \leq \frac{2\cdot \lambda\left(A_n\right)}{\log 2}.\label{eq: lambda gamma Dn}
\end{align}
Hence, it suffices to determine a condition for $\sum_{n=1}^{\infty}\lambda\left(A_n\right)=\infty$.
We first prove \eqref{en: bb 3} as \eqref{en: bb 2} follows from \eqref{en: bb 3} 
by setting for example $c_n=2d_n$.

\emph{ad \eqref{en: bb 3}}:
We have that 
\begin{align*}
 \lambda\left(A_n\right)
 =\frac{1}{d_n}-\frac{1}{ d_n+\left\lfloor d_n/c_n\right\rfloor+1}.
\end{align*}
An easy calculation shows that 
\begin{align*}
 \lambda\left(A_n\right)
 \leq \frac{1}{d_n}-\frac{1}{ d_n+ d_n/c_n+1}
 <\frac{1}{c_nd_n}+\frac{1}{d_n^2}.
\end{align*}
Hence, if $\Gamma$ is finite, 
the first Borel-Cantelli implies $\lambda\left(\limsup_{n\to\infty} A_n\right)=0$.

To prove the second part we first notice that
\begin{align}
 \lambda\left(A_n\right)> \frac{1}{d_n}-\frac{1}{ d_n+ d_n/c_n}>\frac{1}{2c_nd_n},\label{eq: lambda Dn >}
\end{align}
i.e.\  $\sum_{n\in\mathbb{N}}\lambda\left(A_n\right)$ diverges if $\sum_{n\in\mathbb{N}}{1}/\left({c_n d_n}\right)$ does.
Next we assume that $\sum_{n\in\mathbb{N}}1/d_n^2=\infty$. 
Clearly, for all $n\in \N$, we have
\begin{align*}
\left\{a_n=d_n\right\}\subset\left\{d_n\leq a_{n}\leq d_n+\frac{d_n}{c_n}\right\}
\end{align*}
and thus
\begin{align}
\lambda(A_{n})\geq \lambda\left(a_n=d_n\right)=\frac{1}{d_n}-\frac{1}{d_n+1}\geq \frac{1}{2d_n^2},\label{eq: lambda Dn geq}
\end{align}
i.e.\ $\sum_{n=1}^{\infty}\lambda(A_{n})$ diverges if $\sum_{n=1}^{\infty}1/d_n^2$ does.
Since $A_n\in\sigma\left(a_n\right)$, we can apply Lemma \ref{lem: 0-1 mixing}
and obtain the statement of \eqref{en: bb 3}.

\emph{ad \eqref{en: bb 4}}:
We notice that 
\begin{align*}
 \lambda\left(A_n\right)=\frac{1}{d_n+1}-\frac{1}{d_n+\left\lfloor d_n/c_n\right\rfloor+1}.
\end{align*}
An easy calculations shows that 
\begin{align*}
 \lambda\left(A_n\right)\leq \begin{cases}
                              0&\text{if } c_n>d_n\\
                              \frac{1}{d_nc_n}&\text{else}
                             \end{cases}
\end{align*}
and by the first Borel-Cantelli Lemma 
$\sum_{n\colon c_n\leq d_n}1/\left(c_n d_n\right)<\infty$
implies $\lambda\left(\limsup_{n\to\infty} A_n\right)=0$.

For proving the second part  we make use of an analogous statement as in the proof of \eqref{en: bb 3}.
With $A_n\in\sigma\left(a_n\right)$ we have
$\lambda(A_n)>0$   if and only if  $c_n\leq d_n$.
Hence, we are left to show that  $\sum_{n\in\mathbb{N}}\lambda\left(A_n\right)$ diverges if $\sum_{n\colon c_n\leq d_n}1/\left({c_n d_n}\right)$ does. 

First, let us assume that $\left\lfloor d_n/c_n\right\rfloor=1$. Then 
$\lambda\left(A_n\right)=\lambda\left(a_n=d_n\right)>1/(2d_n^2)\geq 1/(2c_nd_n)$ which follows from \eqref{eq: lambda Dn geq}
and the restriction $\left\lfloor d_n/c_n\right\rfloor=1$.

Next, we assume that $d_n/c_n\geq 2$. 
This assumptions yields
\begin{align}
 \lambda\left(A_n\right)
 &> \frac{1}{d_n+1}-\frac{1}{d_n+d_n/c_n}
 >\frac{d_n-c_n}{4c_nd_n^2}\geq \frac{1}{8c_nd_n}.\label{eq: En geq}
\end{align}

Lemma \ref{lem: 0-1 mixing} gives the statement of \eqref{en: bb 4}.
\end{proof}

The proof of Theorem \ref{thm: cor to thm cn dn} needs some extra attention 
and follows as a corollary of Theorem \ref{cn dn}.
\begin{proof}[Proof of Theorem \ref{thm: cor to thm cn dn}]
We first define
\begin{align*}
 c_n'\coloneqq\frac{c_nd_n}{d_n+3c_n}\text{ and }
 c_n''\coloneqq\frac{c_nd_n}{d_n+4c_n}.
\end{align*}
Assume for the following that $\sum_{n\in\N}1/\left(c_n d_n\right)<\infty$. If $c_n\leq d_n/(3+\epsilon)$ we have that
\begin{align*}
 \sum_{n\in\N}\frac{1}{c_n' \left\lfloor d_n\right\rfloor}=\sum_{n\in\N}\frac{1}{\left\lfloor d_n\right\rfloor}\cdot \frac{d_n+3c_n}{c_n d_n}\leq\sum_{n\in\N}\frac{1}{\left\lfloor d_n\right\rfloor}\cdot\frac{2 d_n}{ c_n d_n}=\sum_{n\in\N}\frac{2}{ c_n \left\lfloor d_n\right\rfloor}=\infty.
\end{align*}
Similarly, we obtain if $c_n\leq d_n/(4+\epsilon)$ that
\begin{align*}
 \sum_{n\in\N}\frac{1}{c_n'' \left\lfloor d_n\right\rfloor}=\sum_{n\in\N}\frac{1}{\left\lfloor d_n\right\rfloor}\cdot \frac{d_n+4c_n}{c_n d_n}\leq\sum_{n\in\N}\frac{1}{\left\lfloor d_n\right\rfloor}\cdot\frac{2 d_n}{ c_n d_n}=\sum_{n\in\N}\frac{2}{ c_n \left\lfloor d_n\right\rfloor}=\infty.
\end{align*}
The assumptions $c_n\leq d_n/(3+\epsilon)$ or $c_n\leq d_n/(4+\epsilon)$ respectively and $\sum_{n\in\N}1/\left(c_n d_n\right)<\infty$ imply further that 
$\sum_{n\in\N}1/\left\lfloor d_n\right\rfloor^2<\infty$. The second statement of Theorem \ref{cn dn} implies then that 
\begin{align*}
 \lambda\left(\limsup_{n\to\infty}\left\{\left\lfloor d_n\right\rfloor\leq a_n \leq \left\lfloor d_n\right\rfloor  \left(1+\frac{1}{c_n'}\right)\right\}\right)&=0\text{ or }
 \lambda\left(\limsup_{n\to\infty}\left\{\left\lfloor d_n\right\rfloor\leq a_n \leq \left\lfloor d_n\right\rfloor  \left(1+\frac{1}{c_n''}\right)\right\}\right)=0
\end{align*}
respectively.
By the definition of $c_n'$ and $c_n''$ this implies that 
\begin{align}
 \lambda\left(\limsup_{n\to\infty}\left\{\left\lfloor d_n\right\rfloor\leq a_n \leq \left\lfloor d_n\right\rfloor  \left(1+\frac{1}{c_n}\right)+3\right\}\right)&=0\text{ or }\label{eq: dn an dncn + 1 2a}\\
 \lambda\left(\limsup_{n\to\infty}\left\{\left\lfloor d_n\right\rfloor\leq a_n \leq \left\lfloor d_n\right\rfloor  \left(1+\frac{1}{c_n}\right)+4\right\}\right)&=0\label{eq: dn an dncn + 1 2}
\end{align}
respectively.
From \eqref{r_n} and \eqref{y_n} of Lemma \ref{r y u} we can conclude that
\begin{align*}
\left\{d_n\leq r_n,y_n \leq d_n  \left(1+\frac{1}{c_n}\right)\right\}
&\subset \left\{d_n\leq a_n \leq d_n   \left(1+\frac{1}{c_n}\right)+1\right\}\subset \left\{\left\lfloor d_n\right\rfloor \leq a_n \leq \left\lfloor d_n\right\rfloor   \left(1+\frac{1}{c_n}\right)+3\right\} 
\end{align*}
and from \eqref{u_n} of Lemma \ref{r y u} we can conclude that 
\begin{align*}
\left\{d_n\leq u_n \leq d_n  \left(1+\frac{1}{c_n}\right)\right\}
&\subset \left\{d_n\leq a_n \leq d_n  \left(1+\frac{1}{c_n}\right)+2\right\}
&\subset \left\{\left\lfloor d_n\right\rfloor \leq a_n \leq \left\lfloor d_n\right\rfloor  \left(1+\frac{1}{c_n}\right)+4\right\}.
\end{align*}
Hence, \eqref{eq: dn an dncn + 1 2a} or \eqref{eq: dn an dncn + 1 2} respectively imply
\begin{align*}
 \lambda\left(\limsup_{n\to\infty}\left\{d_n\leq z_n \leq d_n \cdot \left(1+\frac{1}{c_n}\right)\right\}\right)=0,
\end{align*}
where $(z_n)=(r_n)$, $(z_n)=(y_n)$, or $(z_n)=(u_n)$.

In order to prove the infinite part we first define
\begin{align*}
 \overline{c}_n'\coloneqq\frac{c_nd_n}{d_n-3c_n}\text{ and }
 \overline{c}_n''\coloneqq\frac{c_nd_n}{d_n-4c_n}
\end{align*}
and assume for the following that $\sum_{n\in\N}1/\left(c_n d_n\right)=\infty$. If $c_n\leq d_n/(3+\epsilon)$, 
we have that
\begin{align*}
 \sum_{n\in\N}\frac{1}{\overline{c}_n' \left\lceil d_n\right\rceil}
 =\sum_{n\in\N}\frac{1}{\left\lceil d_n\right\rceil}\cdot \frac{d_n-3c_n}{c_n d_n}
 \geq\sum_{n\in\N}\frac{1}{\left\lceil d_n\right\rceil}\cdot \frac{\epsilon}{3+\epsilon}\cdot\frac{d_n}{ c_n d_n}
 =\sum_{n\in\N}\frac{\epsilon}{3+\epsilon}\cdot \frac{1}{c_n \left\lceil d_n\right\rceil}=\infty.
\end{align*}
Similarly, we obtain if $c_n\leq d_n/(4+\epsilon)$ that
\begin{align*}
 \sum_{n\in\N}\frac{1}{\overline{c}_n'' \left\lceil d_n\right\rceil}
 =\sum_{n\in\N}\frac{1}{\left\lceil d_n\right\rceil}\cdot \frac{d_n-4c_n}{c_n d_n}
 \geq\sum_{n\in\N}\frac{1}{\left\lceil d_n\right\rceil}\cdot \frac{\epsilon}{4+\epsilon}\cdot\frac{d_n}{4 c_n d_n}
 =\sum_{n\in\N}\frac{\epsilon}{4+\epsilon}\cdot \frac{1}{c_n \left\lceil d_n\right\rceil}=\infty.
\end{align*}
Applying the other direction of Theorem \ref{cn dn} yields then under the assumption that $c_n\leq d_n/(3+\epsilon)$ or $c_n\leq d_n/(4+\epsilon)$ respectively that
\begin{align*}
 \lambda\left(\limsup_{n\to\infty}\left\{\left\lceil d_n\right\rceil\leq a_n \leq \left\lceil d_n\right\rceil  \left(1+\frac{1}{\overline{c}_n'}\right)\right\}\right)&=1\text{ or }
 \lambda\left(\limsup_{n\to\infty}\left\{\left\lceil d_n\right\rceil\leq a_n \leq \left\lceil d_n\right\rceil  \left(1+\frac{1}{\overline{c}_n''}\right)\right\}\right)=1
\end{align*}
respectively. By the definition of $\overline{c}_n'$ and $\overline{c}_n''$ this implies that 
\begin{align}
 \lambda\left(\limsup_{n\to\infty}\left\{\left\lceil d_n\right\rceil\leq a_n \leq \left\lceil d_n\right\rceil  \left(1+\frac{1}{c_n}\right)-3\right\}\right)&=1\text{ or }\label{eq: dn an dncn - 1 2a}\\
 \lambda\left(\limsup_{n\to\infty}\left\{\left\lceil d_n\right\rceil\leq a_n \leq \left\lceil d_n\right\rceil  \left(1+\frac{1}{c_n}\right)-4\right\}\right)&=1\label{eq: dn an dncn - 1 2}
\end{align}
respectively.
From \eqref{r_n} and \eqref{y_n} of Lemma \ref{r y u} we know that   
\begin{align*}
\left\{d_n\leq r_n,y_n \leq d_n  \left(1+\frac{1}{c_n}\right)\right\}
&\supset \left\{d_n\leq a_n \leq d_n  \left(1+\frac{1}{c_n}\right)-1\right\}\supset \left\{\left\lceil d_n\right\rceil\leq a_n \leq \left\lceil d_n\right\rceil  \left(1+\frac{1}{c_n}\right)-3\right\}
\end{align*}
and from \eqref{u_n} of Lemma \ref{r y u} we know that 
\begin{align*}
\left\{d_n\leq u_n \leq d_n  \left(1+\frac{1}{c_n}\right)\right\}
&\supset \left\{d_n\leq a_n \leq d_n  \left(1+\frac{1}{c_n}\right)-2\right\}\supset \left\{\left\lceil d_n\right\rceil\leq a_n \leq \left\lceil d_n\right\rceil  \left(1+\frac{1}{c_n}\right)-4\right\}.
\end{align*}
Hence, \eqref{eq: dn an dncn - 1 2a} or \eqref{eq: dn an dncn - 1 2} respectively imply
\begin{align*}
 \lambda\left(\limsup_{n\to\infty}\left\{d_n\leq z_n \leq d_n  \left(1+\frac{1}{c_n}\right)\right\}\right)=1,
\end{align*}
where $(z_n)=(r_n)$, $(z_n)=(y_n)$, or $(z_n)=(u_n)$.
\end{proof}

\section{Proof of the central limit theorems}\label{sec: Proof CLT}
We will first prove Theorem \ref{clt}
and show afterwards 
that Theorem \ref{thm: clt an>bn} and Theorem \ref{thm: clt all} can be considered as special cases 
of this theorem. 
To prove Theorem \ref{clt} we 
will use the following lemma
which is a special form of \cite[Theorem 3]{philipp_metrical_1970}.
For the following we set $v_n\ll w_n$ if there exists a constant $K>0$ such that 
$v_n\leq K\cdot w_n$ for all $n\in\mathbb{N}$. 
\begin{lemma}
Let $\left(x_n\right)$ be a sequence of centered random variables with $\sup_{n\in\mathbb{N}}\left\|x_n\right\|_{\infty}\leq 1$
and 
\begin{align*}
 s_n^2\coloneqq\mathbb{E}\left(\left(\sum_{k=1}^nx_k\right)^2\right)\to\infty.
\end{align*}
Denote by $\mathcal{M}_{a,b}$ the $\sigma$-algebra generated by the events $\{x_n <z\}$ with $z\in \mathbb{R}$
and $1 \leq a \leq n \leq b \leq \infty$. Suppose that there exist $\theta\in (0,1)$ such that for all events $A\in\mathcal{M}_{1,r}$ and $B\in\mathcal{M}_{r+n,\infty}$ 
we have
\begin{align}
 \mathbb{P}\left(A\cap B\right)-\mathbb{P}(A)\cdot \mathbb{P}(B)\ll\theta^n\cdot \mathbb{P}(A)\cdot \mathbb{P}(B).\label{eq: mix CLT}
\end{align}
Moreover, assume that
\begin{align}
\sum_{i=m+1}^{m+n}\mathbb{E}\left(\left|x_i\right|\right)\ll\mathbb{E}\left(\left(\sum_{i=m+1}^{m+n}x_i\right)^2\right),\label{eq: E V CLT}
\end{align}
uniformly in $m$.
Then
\begin{align*}
\lim_{n\to \infty}\mathbb{P}\left(\frac{\sum_{i=1}^nx_i}{s_n^2}<z\right)=\frac{1}{\sqrt{2\pi}}\int_{-\infty}^{z} e^{-t^{2}/2}\mathrm{d}t.
\end{align*}
\end{lemma}
\begin{proof}[Proof of Theorem \ref{clt}]
We set $x_n=\mathbbm{1}_{A_n}-\mathbb{E}\left(\mathbbm{1}_{A_n}\right)$. 
That condition \eqref{eq: mix CLT} holds follows from Lemma \ref{lem: mixing cf}.
Hence, we are left to show \eqref{eq: E V CLT},
i.e.\ it suffices to show that there exists $\epsilon'>0$ such that for all $m\in\mathbb{N}_0$, $n\in\mathbb{N}$
\begin{align}
\epsilon'\cdot\sum_{i=m+1}^{m+n}\mathbb{V}\left(\mathbbm{1}_{A_i}\right)< \mathbb{V}\left(\sum_{i=m+1}^{m+n}\mathbbm{1}_{A_i}\right).\label{eq: Vsum sumV}
\end{align}
We first notice that 
\begin{align}
 \mathbb{V}\left(\sum_{i=m+1}^{m+n}\mathbbm{1}_{A_i}\right)
 &=\sum_{i=m+1}^{m+n}\left(\mathbb{V}\left(\mathbbm{1}_{A_i}\right)+\sum_{j=i+1}^{m+n}\mathrm{Cov}\left(\mathbbm{1}_{A_i},\mathbbm{1}_{A_j}\right)\right)\notag\\
 &\geq\sum_{i=m+1}^{m+n}\left(\mathbb{V}\left(\mathbbm{1}_{A_i}\right)-\sum_{j>i}^{m+n}\left|\mathrm{Cov}\left(\mathbbm{1}_{A_i},\mathbbm{1}_{A_j}\right)\right|\right).\label{eq: var 1st estim}
\end{align}
To estimate the last summands we notice that
\begin{align*}
\sum_{j=i+1}^{m+n}\left|\mathrm{Cov}\left(\mathbbm{1}_{ A_i},\mathbbm{1}_{A_{j}}\right)\right|=0,
\end{align*}
if $\gamma\left(A_i\right)=1$.
Assume now that $\gamma\left(A_i\right)< 1$.  
For $i<j$ we have that 
\begin{align}
\left|\mathrm{Cov}\left(\mathbbm{1}_{A_i},\mathbbm{1}_{A_j}\right)\right|
&=\left|\gamma\left(A_i\cap A_j\right)-\gamma\left(A_i\right)\cdot \gamma\left(A_j\right)\right|
\leq\phi_{\gamma}\left(j-i\right)\cdot \gamma\left(A_i\right)\label{eq: estim cov}
\end{align}
and on the other hand
\begin{align}
\left|\mathrm{Cov}\left(\mathbbm{1}_{A_i},\mathbbm{1}_{A_j}\right)\right|
&=\left|\gamma\left(A_i^c\cap A_j\right)-\gamma\left(A_i^c\right)\cdot \gamma\left(A_j\right)\right|
\leq\phi_{\gamma}\left(j-i\right)\cdot \gamma\left(A_i^c\right).\label{eq: estim cova}
\end{align}
Thus, Lemma \ref{lem: phi 1} and Lemma \ref{lem: mixing cf} with an application of \eqref{eq: phi psi} give
\begin{align}
 \sum_{j>i} \left|\mathrm{Cov}\left(\mathbbm{1}_{A_i},\mathbbm{1}_{A_j}\right)\right|
 &\leq \left(\sum_{n=1}^{\infty}\phi_{\gamma}\left(n\right)\right)\cdot\min\left\{\gamma\left(A_i\right),\gamma\left(A_i^c\right)\right\}\notag\\
 &\leq \left(\frac{1-\log 2+\log\log2}{-\log 2}+\left(\frac{\pi^2\cdot\log2}{6}-1\right)\cdot\sum_{n=0}^{\infty} \theta^n\right)\cdot\min\left\{\gamma\left(A_i\right),\gamma\left(A_i^c\right)\right\}\notag\\
 &=\kappa\cdot\min\left\{\gamma\left(A_i\right),\gamma\left(A_i^c\right)\right\},\label{eq: cov calculate}
\end{align}
with $\kappa$ being defined as 
\begin{align*}
\kappa\coloneqq \frac{(\pi^2\log2)/{6}-1}{1-\theta}-\frac{1-\log 2+\log\log2}{\log2}<\frac{1}{2}.
\end{align*}
On the other hand we have that
\begin{align*}
 \mathbb{V}\left(\mathbbm{1}_{A_i}\right)=\gamma\left(A_i\right)\cdot\gamma\left(A_i^c\right)\geq \frac{\min\left\{\gamma\left(A_i\right),\gamma\left(A_i^c\right)\right\}}{2}.
\end{align*}

Hence, we can conclude from \eqref{eq: var 1st estim} that
\begin{align}
 \mathbb{V}\left(\sum_{i=m+1}^{m+n}\mathbbm{1}_{A_i}\right)
 &>\sum_{i=m+1}^{m+n}\left(\frac{\min\left\{\gamma\left(A_i\right),\gamma\left(A_i^c\right)\right\}}{2}-\kappa\cdot\min\left\{\gamma\left(A_i\right),\gamma\left(A_i^c\right)\right\}\right)\notag\\
 &=\left(\frac{1}{2}-\kappa\right)\cdot \sum_{i=m+1}^{m+n}\min\left\{\gamma\left(A_i\right),\gamma\left(A_i^c\right)\right\}\notag\\
 &\geq \left(\frac{1}{2}-\kappa\right)\cdot \sum_{i=m+1}^{m+n}\mathbb{V}\left(\mathbbm{1}_{A_i}\right),\label{eq: V last estim}
\end{align}
proving \eqref{eq: Vsum sumV}.
\end{proof}
\begin{remark}\label{rem: phi(1)}
For this proof we indeed made use of the precise value of the first $\phi$-mixing coefficient to show that $\mathbb{V}\left(\sum_{i=1}^n\mathbbm{1}_{A_i}\right)$ grows at least proportional to $\sum_{i=1}^n\mathbb{V}\left(\mathbbm{1}_{A_i}\right)$.
One can see that the constant $\kappa$ in \eqref{eq: cov calculate}
has to be less than $1/2$ for the factor in %order that $\sum_{i=1}^n\mathbb{V}\left(\mathbbm{1}_{A_i}\right)$ in 
\eqref{eq: V last estim} to be positive.
%can (after the same calculation) be multiplied by a positive factor.

The crucial point is the estimate in \eqref{eq: var 1st estim}.
The covariances appearing there will be estimated with the use of the $\phi$-mixing coefficients as in \eqref{eq: estim cov} and \eqref{eq: estim cova}.
Using the less precise estimate $\phi_{\gamma}\left(1\right)\leq \psi\left(1\right)/2=\log 2-1/2$ 
together with the estimates in Lemma \ref{lem: mixing cf}
is not precise enough to conclude $\sum_{n=1}^{\infty}\phi_{\gamma}\left(n\right)<1/2$.

The first $\phi$-mixing coefficient has already been estimated before.
In \cite{samur_note_1985} Samur showed $\phi_{\gamma}\left(1\right)<1$, which was later
improved by  Philipp to $\phi\left(1\right)\leq 0.4$ in \cite[Lemma 2.1]{philipp_limit_1988}. For this Philipp used 
a weaker estimate on the first $\psi$-mixing coefficient compared to  the one mentioned above.
All these previous estimates are not strong enough to imply
$\sum_{n=1}^{\infty}\phi_{\gamma}\left(n\right)<1/2$.
\end{remark}

\begin{proof}[Proof of Theorem \ref{thm: clt an>bn}]
The theorem follows as a direct application of Theorem \ref{clt}.
What remains to be shown is that \eqref{eq: sum bn} implies 
$\sum_{n=1}^{\infty}\mathbb{V}\left(\mathbbm{1}_{B_n}\right)=\infty$
with $B_n\coloneqq \left\{a_n>b_n\right\}$. 
From \eqref{eq: lambda gamma Dn} it follows that
\begin{align*}
 \sum_{n=1}^{\infty}\mathbb{V}\left(\mathbbm{1}_{B_n}\right)
 \geq \sum_{n=1}^{\infty}\gamma\left(a_n= 1\right)\cdot \gamma\left(B_n\right)
 \geq \frac{\gamma\left(a_1= 1\right)}{\log 2}\cdot \sum_{n=1}^{\infty}\lambda\left(B_n\right)
 =\frac{\gamma\left(a_n= 1\right)}{\log 2}\cdot \sum_{n=1}^{\infty}\frac{1}{b_n+1}
\end{align*}
and thus \eqref{eq: sum bn} implies $\sum_{n=1}^{\infty}\mathbb{V}\left(\mathbbm{1}_{B_n}\right)=\infty$.
\end{proof}

\begin{proof}[Proof of Theorem \ref{thm: clt all}]
As in the proof of Theorem \ref{thm: clt an>bn} we only have to prove 
that the given conditions imply 
\begin{align}
 \sum_{n=1}^{\infty}\mathbb{V}\left(\mathbbm{1}_{A_n}\right)=\infty.\label{eq: sum V An}
\end{align}
\emph{ad \eqref{en: clt 2}:} We have that
\begin{align*}
 \sum_{n=1}^{\infty}\mathbb{V}\left(\mathbbm{1}_{A_n}\right)
 \geq \sum_{n=1}^{\infty}\gamma\left(a_n= 1\right)\cdot \gamma\left(A_n\right)
 \geq \frac{\gamma\left(a_1= 1\right)}{\log 2}\cdot \sum_{n=1}^{\infty}\lambda\left(A_n\right)
 =\frac{\gamma\left(a_n= 1\right)}{\log 2}\cdot \sum_{n=1}^{\infty}\left(\frac{1}{d_n}-\frac{1}{d_n+1}\right)
\end{align*}
and thus $\sum_{n=1}^{\infty}1/d_n^2=\infty$ implies \eqref{eq: sum V An}.

\emph{ad \eqref{en: clt 3}:} We have that
\begin{align*}
 \sum_{n=1}^{\infty}\mathbb{V}\left(\mathbbm{1}_{A_n}\right)
 &=\sum_{n\colon d_n=1}\mathbb{V}\left(\mathbbm{1}_{A_n}\right)
 +\sum_{n\colon d_n>1}\mathbb{V}\left(\mathbbm{1}_{A_n}\right).
\end{align*}
Estimating the first sum gives 
\begin{align*}
 \sum_{n\colon d_n=1}\mathbb{V}\left(\mathbbm{1}_{A_n}\right)
 &\geq \gamma\left(a_n=1\right)\cdot \sum_{n\colon d_n=1}\gamma\left(A_n^c\right)
 \geq \frac{\gamma\left(a_n=1\right)}{\log 2}\cdot \sum_{n\colon d_n=1}\lambda\left(a_n>1+1/c_n\right)\\
 &>\frac{\gamma\left(a_n=1\right)}{\log 2}\cdot \sum_{n\colon d_n=1}\frac{c_n}{2+c_n}
\end{align*}
and $\sum_{n\colon d_n=1}c_n=\infty$ implies \eqref{eq: sum V An}.
Estimating the second sum gives
\begin{align*}
 \sum_{n\colon d_n>1}\mathbb{V}\left(\mathbbm{1}_{A_n}\right)
  \geq \sum_{n\colon d_n>1}\gamma\left(a_n= 1\right)\cdot \gamma\left(A_n\right)
  \geq \frac{\gamma\left(a_1= 1\right)}{\log 2}\sum_{n\colon d_n>1}\lambda\left(A_n\right)
\end{align*}
and by \eqref{eq: lambda Dn >} and \eqref{eq: lambda Dn geq} each of
$\sum_{n\colon d_n>1}1/d_n^2$ and $\sum_{n\colon d_n>1}1/(c_nd_n)$ implies \eqref{eq: sum V An}.

\emph{ad \eqref{en: clt 4}:}
We have that
\begin{align*}
 \sum_{n=1}^{\infty}\mathbb{V}\left(\mathbbm{1}_{A_n}\right)
 \geq \sum_{n\colon c_n>d_n}\gamma\left(a_n= 1\right)\cdot \gamma\left(A_n\right)
 \geq \frac{\gamma\left(a_1= 1\right)}{\log 2}\cdot \sum_{n\colon c_n>d_n}\lambda\left(A_n\right)
\end{align*} 
and by \eqref{eq: En geq} the condition $\sum_{n\colon c_n>d_n}1/(c_nd_n)=\infty$ implies \eqref{eq: sum V An}.
\end{proof}

\section*{Acknowledgements}
We thank Sam Chow for pointing out to us that the value of the first $\phi$-mixing coefficient coincides with the Erd\H{o}s-Ford-Tenenbaum constant.

\begin{thebibliography}{BvdPSZ14}

\bibitem[Ber11]{Be1912b}
F.~Bernstein.
\newblock {\"U}ber eine {A}nwendung der {M}engenlehre auf ein aus der {T}heorie
  der s\"akularen {S}t\"orungen herr\"uhrendes {P}roblem.
\newblock {\em Math. Ann.}, 71(3):417--439, 1911.

\bibitem[Ber12]{Be1912}
F.~Bernstein.
\newblock {\"U}ber geometrische {W}ahrscheinlichkeit und \"uber das {A}xiom der
  beschr\"ankten {A}rithmetisierbarkeit der {B}eobachtungen.
\newblock {\em Math. Ann.}, 72(4):585--587, 1912.

\bibitem[{Bor}09]{Bo1909}
E.~{Borel}.
\newblock {Les probabilit\'es denombrables et leurs applications
  arithm\'etiques.}
\newblock {\em {Rend. Circ. Mat. Palermo}}, 27:247--271, 1909.

\bibitem[Bra05]{bradley_basic_2005}
R.~C. Bradley.
\newblock Basic properties of strong mixing conditions. {A} survey and some
  open questions.
\newblock {\em Probability Surveys}, 2:107--144, 2005.

\bibitem[BvdPSZ14]{borwein_neverending_2014}
J.~Borwein, A.~van~der Poorten, J.~Shallit, and W.~Zudilin.
\newblock {\em Neverending fractions}, volume~23 of {\em Australian
  Mathematical Society Lecture Series}.
\newblock Cambridge University Press, Cambridge, 2014.
\newblock An introduction to continued fractions.

\bibitem[DK02]{dajani_ergodic_2002}
K.~Dajani and C.~Kraaikamp.
\newblock {\em Ergodic Theory of Numbers}.
\newblock The Mathematical Association of America, July 2002.

\bibitem[FLWW09]{fan_khintchine_2009}
A.-H. Fan, L.-M. Liao, B.-W. Wang, and J.~Wu.
\newblock On {K}hintchine exponents and {L}yapunov exponents of continued
  fractions.
\newblock {\em Ergodic Theory Dynam. Systems}, 29(1):73--109, 2009.

\bibitem[For08]{ford_distribution_2008}
K.~Ford.
\newblock The distribution of integers with a divisor in a given interval.
\newblock {\em Ann. of Math. (2)}, 168(2):367--433, 2008.

\bibitem[Gal72]{galambos_remarks_1972}
J.~Galambos.
\newblock Some remarks on the {L}{\"u}roth expansion.
\newblock {\em Czechoslovak Mathematical Journal}, 22(2):266--271, 1972.

\bibitem[IK09]{iosifescu_metrical_2009}
M.~Iosifescu and C.~Kraaikamp.
\newblock {\em Metrical Theory of Continued Fractions}.
\newblock Springer Netherlands, 2002. edition, December 2009.

\bibitem[JP18]{jenkinson_rigorous_2018}
O.~Jenkinson and M.~Pollicott.
\newblock Rigorous effective bounds on the {H}ausdorff dimension of continued
  fraction {C}antor sets: {A} hundred decimal digits for the dimension of
  {$E_2$}.
\newblock {\em Adv. Math.}, 325:87--115, 2018.

\bibitem[Khi35]{khintchine_metrische_1935}
A.~Y. Khintchine.
\newblock Metrische {K}ettenbruchprobleme.
\newblock {\em Compositio Math}, 1:361--382, 1935.

\bibitem[Khi64]{khintchine_continued_1964}
A.~Ya. Khintchine.
\newblock {\em Continued fractions}.
\newblock The University of Chicago Press, Chicago, Ill.-London, 1964.

\bibitem[KMS16]{kesseboehmer_infinite_2016}
M.~Kesseb\"ohmer, S.~Munday, and B.~O. Stratmann.
\newblock {\em Infinite Ergodic Theory of Numbers}.
\newblock De Gruyter Graduate. De Gruyter, Berlin, 2016.

\bibitem[KS07]{KS}
Kesseb\"ohmer, M., Stratmann, B.O.: 
\newblock A multifractal analysis for Stern-Brocot intervals, continued fractions and Diophantine growth rates.
\newblock {\em J. Reine Angew. Math.}   605:133--163,  2007.

\bibitem[KZ06]{TexanConjecture}
Kesseb\"ohmer, M., Zhu, S.: 
\newblock Dimension sets for infinite IFSs: the Texan conjecture.
\newblock {\em J. Number Theory}  116(1):230--246,  2006.
%\newblock \urlprefix\url{https://doi.org/10.1016/j.jnt.2005.04.002}

\bibitem[LR16]{liau_upper_2016}
L.~Liao and M.~Rams.
\newblock Upper and lower fast {K}hintchine spectra in continued fractions.
\newblock {\em Monatsh. Math.}, 180(1):65--81, 2016.

\bibitem[Phi67]{philipp_metrical_1967}
W.~Philipp.
\newblock Some metrical theorems in number theory.
\newblock {\em Pacific Journal of Mathematics}, 20(1):109--127, 1967.

\bibitem[Phi70]{philipp_metrical_1970}
W.~Philipp.
\newblock Some metrical theorems in number theory. {II}.
\newblock {\em Duke Math. J.}, 37:447--458, 1970.

\bibitem[Phi88]{philipp_limit_1988}
W.~Philipp.
\newblock Limit theorems for sums of partial quotients of continued fractions.
\newblock {\em Monatshefte f{\"u}r Mathematik}, 105(3):195--206, 1988.

\bibitem[PW73]{philipp_invariance_1973}
W.~Philipp and G.~R. Webb.
\newblock An invariance principle for mixing sequences of random variables.
\newblock {\em Z. Wahrscheinlichkeitstheorie und Verw. Gebiete}, 25:223--237,
  1973.

\bibitem[Sam85]{samur_note_1985}
J.~D. Samur.
\newblock A note on the convergence to {G}aussian laws of sums of stationary
  {$\phi$}-mixing triangular arrays.
\newblock In {\em Probability in {B}anach spaces, {V} ({M}edford, {M}ass.,
  1984)}, volume 1153 of {\em Lecture Notes in Math.}, pages 387--399.
  Springer, Berlin, 1985.

\bibitem[WW08]{wang_hausdorff_2008}
B.-W. Wang and J.~Wu.
\newblock Hausdorff dimension of certain sets arising in continued fraction
  expansions.
\newblock {\em Adv. Math.}, 218(5):1319--1339, 2008.

\bibitem[Zup81]{zuparov_theorem_1981}
T.~M. Zuparov.
\newblock On a theorem from the metric theory of continued fractions.
\newblock {\em Izv. Akad. Nauk UzSSR Ser. Fiz.-Mat. Nauk}, (6):9--12, 78, 1981.

\end{thebibliography}


\end{document}